\documentclass[oneside,a4paper,11pt]{amsart}

\usepackage{amssymb,amsmath,t1enc,amsthm,geometry,enumitem}
\usepackage[utf8]{inputenc}
\usepackage[english]{babel}
\usepackage{url}
\usepackage{xcolor}
\usepackage{enumitem}

\title[Almost all Finsler metrics have infinite dimensional holonomy
group]{Almost all Finsler metrics have \\ infinite dimensional holonomy group}

\author[Balázs Hubicska]{B.~Hubicska}
\author[Vladimir S. Matveev]{V.S.~Matveev}
\author[Zolt\'an Muzsnay]{Z.~Muzsnay}

\address{Balázs Hubicska: \vspace{3pt}}
\email{\url{hubicska.balazs@science.unideb.hu} \vspace{-7pt}}
\address{University of Debrecen, Institute of Mathematics, Pf.~400, Debrecen,
  4002, Hungary}

\address{Vladimir Matveev: \vspace{3pt}}
\email{\url{vladimir.matveev@uni-jena.de} \vspace{-7pt}}
\address{Friedrich-Schiller-Universität, Institut für Mathematik, 07737 Jena,
  Germany}

\address{Zolt\'an Muzsnay: \vspace{3pt}}
\email{\url{muzsnay@science.unideb.hu} \vspace{-7pt}} \address{University of
  Debrecen, Institute of Mathematics, Pf.~400, Debrecen, 4002, Hungary}

\newcommand{\halmazpont}[2]{\left\{\,#1 : #2\,\right\}}

\newcommand{\Hol}{\mathcal{H}ol} 
\newcommand{\hol}{\mathfrak{hol}\hspace{1pt}} 
\newcommand{\ihol}{\mathfrak{hol}^*}

\newcommand{\R}{\mathbb R}
\newcommand{\N}{\mathbb N}
\newcommand*{\p}{\mathcal{P}}
\newcommand*{\I}{\mathcal{I}}
\newcommand*{\E}{\mathcal{E}}
\newcommand*{\B}{\mathbb{B}}
\newcommand*{\F}{\mathcal{F}}
\newcommand*{\T}{\widehat {T}}
\newcommand*{\X}[1]{{\mathfrak X}\left(#1\right)}
\newcommand{\ts}{\textsuperscript}

\newcommand{\diff}[1]{\mathcal{D}i\!f\!\!f(#1)}
\newcommand{\diffp}[1]{\mathcal{D}i\!f\!\!f_+(#1)}

\def\={\!=\!}

\theoremstyle{plain}
\newtheorem{theorem}{Theorem}[section]
\newtheorem{proposition}[theorem]{Proposition}

\newtheorem{property}[theorem]{Property}

\theoremstyle{definition}
\newtheorem{definition}[theorem]{Definition}
 
\newtheorem{remark}[theorem]{Remark}

\begin{document}

\begin{abstract}
  We show that the set of Finsler metrics on a manifold contains an open
  everywhere dense subset of Finsler metrics with infinite-dimensional
  holonomy groups.
\end{abstract}

\maketitle

\bigskip  

  
\medskip
  

\bigskip \bigskip

\section{Introduction}
\label{sec:1}

Finsler metrics appeared already in the inaugural lecture of B.~Riemann in
1854 \cite{Riemann_1854}, under the name \textit{generalized metric}.  At the
beginning of the XXth century, the intensive study of Finsler metrics was
motivated by the optimal transport theory.  A group of mathematicians lead by
C.~Cartheodory aimed to adapt mathematical tools which were effective in
Riemannian geometry (such as affine connections, Jacobi vector fields,
sectional curvature) for a more general situation. P.~Finsler was a student of
Cartheodory and his dissertation \cite{Finsler_1918} is one of the important
steps on this way.

Riemannian geometry is one of the main sources of challenging problems in
Finsler geometry: many Riemannian results are not valid in the Finslerian
setup and one asks under what additional assumptions they are correct.

Our paper studies the holonomy groups of Finsler manifolds. We give precise
definitions later; at the present point let us recall that the Berwald
connection (introduced by L.~Berwald in 1926 \cite{Berwald_1926}) can be
viewed as an Ehresmann connection on the unit tangent bundle $\mathcal{I}M$.
Its holonomy group (at $x\in M$) is the subgroup of the group
$\diff{\mathcal{I}_x}$ generated by the parallel transports along the loops
starting and ending at $x$.

For Riemannian metrics, the Berwald connection specifies to the Levi-Civita
connection.  Study of Riemannian holonomy groups is a prominent topic in
Riemannian geometry and mathematical physics.  It is known (see e.g. A.~Borel
and A.~Lichnerowicz \cite{Borel_Lichn_1952}) that the holonomy group is a
subgroup of the orthogonal group; in particular it is always
finite-dimensional.  Moreover, all possible holonomy groups are described and
classified due in particular to breakthrough results of M.~Berger and
J.~Simons \cite{Berger_1955, Simons}.

In the Finslerian case, the situation is very different and not much is known.
By \cite{Szabo_1981} (see also \cite{Ma2,Matveev1,Vi}) the so-called Berwald
manifolds have finite-dimensional holonomy group. Also the so-called Landsberg
manifolds have finite-dimensional holonomy group \cite{Kozma_2000}, but it is
not jet known whether nonberwaldian Landsberg manifolds exist \cite{Matveev}.
We are not aware of other examples of Finsler metrics with finite-dimensional
holonomy group, it is an interesting problem to find such.
	
From the other side, there are also not many explicit examples of Finsler
manifolds with infinite-dimensional holonomy group
\cite{Hubicska_Muzsnay_tangent_2017, Muzsnay_Nagy_max_2015}, and all these
examples have constant curvature.  A natural and fundamental question in this
context is whether for a generic Finsler manifold the holonomy group is
infinite-dimensional. Its simplest version was explicitly asked by S.-S.~Chern
et al in \cite[page 85]{ChernShen2005}.

In our paper we prove that for a generic Finsler manifold the holonomy group
is infinite-dimensional:

\begin{theorem}
  \label{main_thm_2}
  In the set $\F$ of $C^\infty$-smooth Finsler metrics on a manifold $M$ of
  dimension $n\geq 2$, there exists a subset $\widetilde{\F}$ of Finsler
  metrics with infinite dimensional holonomy group, which is open and
  everywhere dense in any $\mathcal{C}^{m}$-topology, $m\geq 8$.
\end{theorem}
  
What we essentially prove is that one can $C^\infty$ small perturb any
Finsler metric $F$ at a neighbourhood of any point $x\in M$ such that for
every open nonempty subset $U \subseteq M$ containing $x$, the perturbed
Finsler structure $(U, F_t)$ has infinite dimensional holonomy
group. Moreover, the metric $F_t$ has an open subset in the space of all
$C^\infty$-Finsler metrics equipped with $C^{m\ge 8}$-topology such that
every Finsler metric $F''$ in this set also has infinite dimensional
holonomy group.  Theorem \ref{main_thm_2} is true mircrolocally and on the
level of germs (see Remark \ref{rem:germs}). The perturbation is given by a
formula.  We show that for almost every $t \in [0,1]$, the perturbation
$F_t$ on its indicatrix $\I_x^{\, t}$ has the full infinity-jet at every
point $y\in \I_x^{\, t}$. Based on this, we conjecture that in the generic
case, the holonomy group of a Finsler manifold coincides with the full
diffeomorphism group of the indicatrix.

Our results imply that, in contrast to the Riemannian case, the closure of
the holonomy group is not a compact group for most Finlser metric.  Similar
results for the \emph{linear holonomy group} (defined via the linear
parallel transport) were recently obtained in \cite{Ivanov_Lytchak_2019}.

The proof is organised as follows. We first show that the standard Funk
metric $F_{Funk}$ has `sufficiently large' holonomy algebra. For dimension
2, this was known \cite{Hubicska_Muzsnay_tangent_2017,
  Muzsnay_Nagy_max_2015}, we generalise these results to all
dimensions. Next, we employ the trick from \cite[\S 3.1]{relativity} and
show that with the help of $F_{Funk}$ one can perturb an arbitrary Finsler
metric such the result also has `sufficiently large' holonomy
algebra. Then, we show that if the holonomy algebra is `sufficiently large'
then it is infinite-dimensional. This step is in fact a general statement
about algebras of smooth vector fields and possibly can be applied
elsewhere. Therefore, let us formulate it as Theorem
\ref{thm:no_transitive_finite} below, this will also explain what we
understand by `sufficiently large'.

An algebra $\mathfrak{g}$ of vector fields on $U\subseteq \mathbb{R}^n$ is
called \emph{$3$-jet generating} at $x\in U$, if the set of 3rd jets of these
vector fields at the point $x$ coincides with the space of all 3rd jets of
vector fields at $x$ (see Definition \ref{def:gen}).  In other words, every
vector field can be approximated at $x$ with order three by a vector field
from the algebra. For example, if the algebra is \emph{locally transitive} at
$x$, i.e., if the elements of the algebra at $x$ span the whole $T_xU$, then
it is $0-$jet generating.

\begin{theorem}
  \label{thm:no_transitive_finite}
  Let $\mathfrak{g}$ be a Lie algebra of vector fields on a manifold $U$. If
  there exists a point where it is 3-jet generating, then $\mathfrak{g}$ is
  infinite-dimensional.
\end{theorem}

If dimension $U$ is 1, the result is known and is due to Sophus Lie, see
e.g.~\cite[Theorem 2.70]{Olver}. As examples show (see e.g.~the tables at the
back of \cite{Olver} where vector field algebras of arbitrary finite dimension
are listed), the 3-jet generating property is important.

\subsection*{Acknowledgement} The cooperation of V.M. with B.H. and Z.M. was
started within and was supported by the exchange program of DAAD (PPP Ungarn;
project number 57447379) and TKA (project number 307818), ÚNKP-19-3-I-DE-531
New National Excellence Program of the Ministry for Innovation and Technology.
V.M. thanks the University of Debrecen, B.H. and Z.M. thank the
Friedrich-Schiller-Universität for their hospitality. We are very grateful to
Boris Kruglikov and Peter Olver for useful discussions.  \bigskip

\section{Preliminaries}
\label{sec:2}

\bigskip

Let $M$ be an $n$-dimensional manifold, $TM$ its tangent manifold,
$\pi: TM \to M$ the canonical projection.  Local coordinates $(x_i)$ on $M$
induce local coordinates $(x_i, y_i)$ on $TM$.  The $k$th order jet of a
function $f\in C^\infty(M)$ (resp. smooth vector field $V\in \mathfrak X(M)$)
at $x\in M$ will be denoted by $j^k_x(f)$ (resp.~$j^k_x(V)$). In local
coordinates, the $k$th order jet can be viewed as the collection of all
derivatives of the function or of the vector field up to the order $k$. For
example, the $1$st order jet of a function at a point $x$ can be viewed as
$n+1$ numbers
\begin{math}
  \left(f, \tfrac{\partial f}{\partial x_1}, \dots , \tfrac{\partial
      f}{\partial x_n} \right),
\end{math}
and of a vector field $V$ as $n(n+1)$ numbers
\begin{math}
  \left(V_1, \dots,V_n, \tfrac{\partial V_1}{\partial x_1}, \dots ,
    \tfrac{\partial V_n}{\partial x_n}\right).
\end{math}
\medskip

\subsection{Finsler manifolds, connection}
\label{sec:finsl}   
 
\medskip

\noindent
The function $F \colon TM \to \mathbb{R}_+$ is called a \textit{Finsler
  metric}, if it is a positively 1-homogeneous continuous function,
$C^\infty$-smooth on $\T M=TM\!\setminus\!\{0\}$ and
\begin{displaymath}
\label{metric_coeff}
  g_{ij}= \frac{\partial^2 \E}
  {\partial y_i\partial y_j}
\end{displaymath}
is positive definite at every $y\in \T_xM$, where $\E:=\frac{1}{2}F^2$
denotes the \textit{energy function} of $F$. A pair $(M, F)$
is called a \textit{Finsler manifold}. The hypersurface of $T_xM$
defined by
\begin{equation}
  \label{eq:2}
  \I_x \! =\! \halmazpont{y \in T_xM}{F_x(y) \! = \! 1}
\end{equation}
is called the \emph{indicatrix} at $x \in M$.  The \emph{geodesics} of a
Finsler manifold $(M, F)$ are given by the solutions of the following
system of second order ordinary differential equations
\begin{equation}
  \label{eq:geodesic}
  \ddot{x}^i + 2G^i(x,\dot x)=0,
\end{equation}
where the geodesic coefficients $G^i=G^i(x,y)$ are determined by the
formula
\begin{equation}
  \label{eq:G_i}
  \hphantom{\qquad i = 1,\dots, n,}   
  G^i = \frac{1}{4} g^{il}\Bigl(2\frac{\partial  g_{jl}}{\partial x_k} 
  - \frac{\partial g_{jk}}{\partial x_l}  \Bigr) \, y_jy_k,
  \qquad i = 1,\dots, n.
\end{equation}
The (Berwald) parallel translation on a Finsler manifold can be introduced by
considering the Ehresmann connection: the horizontal distribution is
determined by the image of the horizontal lift $T_xM\to T_{(x,y)}TM$
defined in a local basis as
\begin{equation}
  \label{eq:lift2}
  \delta_i:= \left(\frac{\partial}{\partial x_i} \right)^{\!\! h} =
  \frac{\partial}{\partial x_i} -G^k_i  \frac{\partial}{\partial y_k},
\end{equation}
where $y\in \T_xM$ and
\begin{math}
  G^i_j = \frac{\partial G^i}{\partial y_j}.
\end{math}
We have the decomposition
\begin{displaymath}
  TTM = \mathcal{H} \oplus \mathcal{V},
\end{displaymath}
where $\mathcal{V}=\ker \pi_*$ is the vertical distribution. The corresponding
projectors are denoted by $h$ and $v$. The \textit{horizontal Berwald
  covariant derivative} of a vertical vector field $\xi$ with respect to a
vector field $X\in \mathfrak{X}(M)$ is defined by
\begin{displaymath}
\nabla_{X} \xi = [X^h , \xi ].
\end{displaymath}
In local coordinates, if
\begin{math}
  \xi=\xi^{i}(x,y) \frac{\partial}{\partial y_{i}}
\end{math}
and $X(x)=X^{i} \frac{\partial}{\partial x_{i}}$, then 
\begin{displaymath}
  \label{eq:covder}
  \nabla_{X} \xi = \left( \frac{\partial \xi ^{i}}{\partial x_{j}}
    - G^{k}_j  \frac{\partial \xi ^{i}}{\partial y_{k}}
    +\frac{\partial G^{i}_j }{\partial y_{k}} \xi^{k}  \right)
  X_{j} \frac{\partial}{\partial y_i}.
\end{displaymath}

\bigskip

\subsection{Parallel translation and  curvature}
\label{sec:finsler2}  \ 

\medskip

\noindent
Parallel vector fields along a curve $c$ are characterized by the property
that their covariant derivative vanishes. Parallel translation can be
obtained through the following geometric construction: the horizontal lift
of a curve $c\colon [0,1]\to M$ with initial condition $X_0\in T_{c(0)}M$ is
a curve $c^h\colon [0,1]\to TM$ such that $\pi \circ c^h \= c$,
$\frac{d c^h}{dt}\=(\frac{d c}{dt})^h$ and $c^h(0) \= X_0$. Then the
parallel translation of $X_0$ along the curve $c$ from $c(0)$ to $c(1)$ is
\begin{equation}
   \label{eq:parallel_3}
   \p_{c}(X_0) = c^h(1). 
\end{equation}
The horizontal distribution $\mathcal{H}$ is, in general, non-integrable. The
obstruction to its integrability is given by the \emph{curvature tensor}
\begin{math}
   R=\frac{1}{2}[h,h]
\end{math}
which is the Nijenhuis torsion of the horizontal projector $h$ associated to
the subspace $\mathcal{H}$. The \emph{curvature tensor} field is defined by
\begin{equation}
  \label{eq:R}
  R = R^i_{jk}(x,y) \, dx_j \! \otimes dx_k \! \otimes \frac{\partial}{\partial
    y_i}
\end{equation}
where
\begin{equation}
\label{eq:curv_tensor}
  R^i_{jk} =  \frac{\partial G^i_j}{\partial x_k}
  - \frac{\partial G^i_k}{\partial x_j} + G_j^m G^i_{k m}
  - G_k^m G^i_{j m}
\end{equation} 
in a local coordinate system with
$G^i_{jk}=\frac{\partial G^i_j}{\partial y_k}$.

\begin{remark}
  \label{remark:jet}
  From formula \eqref{eq:G_i} we get that the geodesic coefficients
  $G^i(x,y)$ can be calculated in terms of the 3\ts{rd} order jet of the
  Finsler function $F$ at $(x,y)$. Therefore, the coefficients
  $R^i_{jk}(x,y)$ of the curvature tensor and the curvature vector fields
  \begin{equation}
    \label{eq:R_map}
    \mathcal R_{ij} := R(\delta_i,\delta_j),  \qquad i,j =1, \dots, n, 
  \end{equation}
  can be expressed algebraically by the 5\ts{th} order jet of $F$. More
  generally, the value of $k$\ts{th} order covariant derivatives and
  $k$\ts{th} successive Lie brackets of curvature vector fields can be
  expressed algebraically by the $(k+5)$\ts{th} order jet of $F$.
\end{remark}

\bigskip

\subsection{The holonomy group, the holonomy algebra and its subalgebras}
\label{sec:holonomy} \

\medskip

\noindent
The \emph{holonomy group} $\Hol_x(M,F)$ of a Finsler manifold $(M, F)$ at
a point $x\in M$ is the group generated by parallel translations along
piece-wise differentiable closed curves starting and ending at $x$. Since
the parallel translation \eqref{eq:parallel_3} is 1-homogeneous and preserves
the norm, one can consider it as a map
\begin{equation}
  \label{eq:hol_elem}
  \p_{c}:\I_x \to \I_x;
\end{equation}
therefore, the holonomy group can be seen as a subgroup of the
diffeomorphism group of the indicatrix:
\begin{displaymath}
  \Hol_x (M, F) \ \subset \ \diff{\I_x}, 
\end{displaymath}
and its tangent space at the identity, is called the holonomy algebra:
\begin{equation}
  \label{eq:hol_x}
  \hol_x (M, F) \ \subset \  \X{\I_x}.
\end{equation}
We are listing below the most important properties of the \emph{holonomy
  algebra} (see \cite{Hubicska_Muzsnay_tangent_2017}):
\begin{property}
  \label{prop:hol} \
  \begin{enumerate}[topsep=3pt, partopsep=3pt,leftmargin=25pt]
    \setlength{\itemsep}{1pt} \setlength{\parskip}{1pt}
    \setlength{\parsep}{1pt}
  \item \label{it:1} $\hol_x (M, F)$ is a Lie subalgebra of $\X{\I_x}$,
  \item \label{it:3} the exponential image of $\hol_x (M, F)$ is in the
    topological closure of $\Hol_x(M,F)$.
  \end{enumerate}
\end{property}
The \emph{infinitesimal holonomy algebra} $\ihol_x(M, F)$ is generated by
curvature vector fields and their horizontal Berwald covariant derivatives,
that is:
\begin{equation}
  \label{eq:inf_hol_alg}
  \ihol_x(M, F):=
  \left\langle
    \nabla_{Z_1} \dots \nabla_{Z_k} R(X^h, Y^h) \ \big| \ X,Y, Z_1, \dots,
    Z_k\in \X{M}  \right\rangle_{Lie}.
\end{equation}
The infinitesimal holonomy algebra $\mathfrak{hol}_x^{*}(M, F)$ is a Lie
subalgebra of $\hol_x (M, F)$.

\begin{remark}
  \label{rem:infinitesimal}
  The infinitesimal holonomy algebra is local in nature, that is for any
  open neighbourhood $U$ of $x\in M$ we get
  \begin{math}
    \ihol_x(M, F) = \ihol_x \bigl( U, F|_{\pi^{-1}(U)} \bigr).
  \end{math}
  For that reason, we will simplify the notation
  \begin{displaymath}
    \ihol_x(F) := \ihol_x(M, F)
  \end{displaymath}
  by omitting the neighborhood of the point where the infinitesimal holonomy
  algebra is determined. Indeed, the curvature vector fields, their horizontal
  Berwald covariant derivatives and their Lie brackets can be computed on an
  arbitrarily small neighbourhood of $x$, therefore their value at $x$ can be
  determined locally.
\end{remark}

We have the inclusions of Lie algebras:
\begin{equation}
  \label{eq:curv_hol_x}
  \mathfrak{hol}_x^{*}(F) \subset  \hol_x (M, F) \subset \X{\I_x}, 
\end{equation}
therefore, at the level of groups, we get
\begin{equation}
  \label{eq:group_curv_hol_x}
  \exp \bigl(\mathfrak{hol}_x^{*}(F) 
  \bigr) \subset   \exp \bigl(\hol_x (M, F)\bigr)
  \subset \Hol_x^{c} \bigl(M, F) \subset \diff{\I_x}
\end{equation}
where $\Hol_x^c (M)$ denotes the topological closure of the holonomy
group with respect to the $C^\infty$--topology of $\diff{\I_x}$. We call
a Lie algebra infinite dimensional if it contains infinitely many
$\R-$linearly independent elements. Clearly, using the tangent property
of the holonomy algebra, if $\hol_x(M, F)$ is infinite dimensional, then
the holonomy group cannot be a finite dimensional Lie group. This
observation motivates the following
\begin{definition}
  \label{def:inf_dim}
  The holonomy group of a Finsler manifold $(M, F)$ is called infinite
  dimensional if its holonomy algebra is infinite dimensional. 
\end{definition}
We refer to \cite{Hubicska_Muzsnay_tangent_2017} for a discussion of the
tangent Lie algebras of diffeomorphism groups and of the relation between the
holonomy group and the holonomy algebra.

\bigskip

\section{On the holonomy of the standard Funk metric}
\label{sec:3}

\noindent
A Funk metric can be described as follows. Let $\Omega$ be a bounded convex
domain in $\mathbb{R}^n$ and denote its boundary by $\partial \Omega$. We
can define a Finsler norm function $F_\Omega (x,y)$ in the interior of
$\Omega$ for any vector $y\in T_x\Omega$ by the formulas
\begin{displaymath}
  F_\Omega(x,y)>0, \qquad x+\frac{y}{F_\Omega(x,y)}=z,
\end{displaymath}
where $z\in \partial \Omega$. This norm function is called the Funk norm
function induced by $\Omega$.  The Funk norm induced by the origo centered
unit ball $B^n\subset \R^n$ will be called the \emph{standard Funk norm}
and will be denoted by $F_{\B^{^n}}$. We denote by $o=(0, \dots, 0)$ the
origin in $\R^n$.

\begin{remark}
  \label{rem:dif_group}
  The holonomy of $(\B^{^2}, F_{\B^{^2}})$ was investigated in
  \cite[Chapter 5]{Muzsnay_Nagy_max_2015}. It was proved that the
  infinitesimal holonomy algebra $\ihol_o (F_{\B^{^2}})$ contains the
  Fourier algebra $\mathsf{F}(\mathbb S^1)$ whose elements are vector
  fields $f\frac{d}{dt}$ such that $f(t)$ has finite Fourier series. One
  has
  \begin{equation}
    \label{eq:fourier_hol}
    \mathsf{F}(\mathbb S^1)  \ \subset \ \ihol_o (F_{\B^{^2}}) \
    \subset \ \X{\mathbb S^1}. 
  \end{equation}
  Since $\mathsf{F}(\mathbb S^1)$ is dense in $\X{\mathbb S^1}$, we get the
  same from \eqref{eq:fourier_hol} for $\ihol_o(F_{\B^2})$.  Using the
  exponential map, one can obtain from \eqref{eq:group_curv_hol_x} that the
  closure of the holonomy group of the Finsler surface $(\B^2, F_{\B^2})$
  is $\diffp{\mathbb S^1}$, the group of orientation preserving
  diffeomorphisms of the circle \cite[Theorem 5.2]{Muzsnay_Nagy_max_2015}.
\end{remark}

\begin{proposition}
  \label{prop:ind_vf}
  The infinitesimal holonomy algebra of the standard Funk metric
  $F_{\B^{^n}}$ at $o \in \R^n$ is infinite dimensional.
\end{proposition}

\begin{proof}
  For $n=2$ the proof follows directly from Remark \ref{rem:dif_group}
  since \eqref{eq:fourier_hol} shows that $\ihol_o (F_{\B^{^2}})$ contains
  the infinite dimensional Lie algebra $\mathsf{F}(\mathbb S^1)$.

  Let us consider the $n>2$ case. For each tangent $2$-plane
  $\mathcal K \subset T_o\B^n$ the restriction of $F_{\B^{^n}}$ to
  $\B^2:= \B^n \cap \mathcal K$ is the standard Funk metric on $\B^2$. One
  can suppose that $\mathcal K$ is the 2-plane generated by
  $\frac{\partial}{\partial x_1}$ and $\frac{\partial}{\partial
    x_2}$. Then, using the totally geodesic property, the curvature vector
  field and its successive covariant derivatives with respect to the
  directions $\frac{\partial}{\partial x_1}$ and
  $\frac{\partial}{\partial x_2}$ on $\B^2$ at $o$ can be inherited as
  restriction of the corresponding vector fields of $\B^n$.  Consequently,
  the elements of the Fourier algebra can be obtained as the restriction of
  elements of the infinitesimal holonomy algebra $\ihol(F_{\B^n})$ and we
  have
  \begin{displaymath}
    \mathsf{F}(\mathbb S^1)  \ \subset \ \ihol_o (F_{\B^{^2}})
    \ \simeq \ \ihol_o (F_{\B^{^n}}\big|_{\B^{^n}\cap \mathcal K}) \
    \subset \ \ihol_o (F_{\B^{^n}}).
  \end{displaymath}
  It follows that $\ihol_o(F_{\B^{^n}})$ contains infinitely many
  $\R$-independent vector fields which can be expressed by the curvature
  vector fields and their covariant derivatives.
\end{proof}

\begin{definition}
  \label{def:gen}
  A set  $\mathcal{V} \subset \X{M}$ of vector fields on a
  manifold $M$ is called
  \begin{enumerate}
  \item $k$-jet generating at $x\in M$ if the natural map
    \begin{math}
      j^k_{x}\colon \mathcal{V} \to J^k_{x} (\X{M})
    \end{math}
    is surjective,
  \item jet generating on $M$ if at any $x\in M$ and for any $k\geq 0$ it is
    $k$-jet generating.
  \end{enumerate}
\end{definition}
  
We have the following
\begin{proposition}
  \label{prop:Funk_generating}
  The infinitesimal holonomy algebra $\ihol_o (F_{\B^{^n}})$ of the standard
  Funk metric at the point $o\in \B^{^n}$ has the jet generating property on
  the indicatrix $\mathcal{I}_o$.
\end{proposition}

\begin{proof}
  According to Definition \ref{def:gen}, we have to show that for any
  $y \in \I_o$ and $k\in \N$ the jet-projection
  \begin{math}
    \ihol_o(F_{\B^{^n}}) \rightarrow J_y^k(\X{\I_o})
  \end{math}
  is onto. 

  In the case $n=2$, we get from Remark \ref{rem:dif_group} that
  $\ihol_o (F_{\B^{^2}})$ is dense in $\X{\I_o}$. It follows that the
  restriction of the $k$\ts{th} order jet projection on the infinitesimal holonomy
  algebra
  \begin{equation}
    \label{eq:2_dim_jet_onto}
    j^k_y: \ihol_o(F_{\B^2}) \longrightarrow J^k_y(\X{\I_o}),
  \end{equation}
  is onto. In other words, any given $k$\ts{th} order jet in $J^k_y(\X{\I_o})$
  can be realized as the $k$-jet of an element of the infinitesimal holonomy
  algebra. Clearly we have the jet generating property.
    
  Let us consider the $n>2$ case. If $y\in \I_o$ and $v\in T_y(\I_o)$, let
  $\mathcal K_{y,v}$ be the 2-plane determined by these vectors. Using the
  argument of the proof of Proposition \ref{prop:ind_vf} and
  \eqref{eq:2_dim_jet_onto} we get that
  \begin{equation}
    j^k_y: \ihol_o(F_{\B^n}\big|_{\I_o\cap \mathcal K_{y,v}})
    \longrightarrow   J^k_y(\X{\I_o\cap \mathcal K_{y,v}})
  \end{equation}
  is onto. It follows that for $y \in \I_o$ and $v\in T_y(\I_o)$, any
  $k$\ts{th} order $v$-directional derivative can be realised by elements of
  the holonomy algebra. Using the local coordinate system
  $y_1, \dots, y_{n-1}$ on the $n-1$-dimensional indicatrix $\I_o$ we get that
  for any given $(z, z_1, \dots, z_k) \in (\R^{(n-1)})^{(k+1)}$ there exist
  $\xi\in \ihol_o(F_{\B^n})$ such that
  \begin{equation}
    \label{eq:derivatives}
    \xi\big|_y=z, \quad 
    (\mathcal D_v \xi)\big|_y=z_1, \quad \dots \quad 
    (\mathcal D^{(k)}_v \xi)\big|_y=z_k,
  \end{equation}
  where we consider a locally constant extension of $v$ when the higher order
  derivatives are computed.  For the completion of the proof however we must
  be able to generate all $k$-th jet at $y$, that is the terms corresponding
  to mixed partial derivatives as well. This is possible, by using higher
  order derivatives corresponding to several directions.  Indeed, one can use
  the polarization technique to show that the $k$\ts{th} order mixed partial
  derivatives are determined by the $k$\ts{th} order directional derivatives,
  a similar way as the quadratic form determines the corresponding symmetric
  bilinear form, or more generally, as the homogeneous from of degree $k$
  determines the corresponding symmetric multilinear $k$-from.  Indeed,
  considering any $v_1, \dots v_k \in T_y (\I_o)$ and their constant extension
  in a neighbourhood of $y$, we get
  \begin{equation}
    \label{eq:mixed_der}
    \mathcal  D_{v_1}\bigl(\mathcal  D_{v_2} \cdots
    (\mathcal D_{v_k} \xi)\bigr) = \frac{1}{k!} \sum^k_{s=1} \sum_{\ 1 \leq
      j_1 < \dots < j_s \leq k} (-1)^{k-s} \mathcal D^{(k)}_{v_{j_1}+ \dots +
      v_{j_s}} \xi.
  \end{equation}
  It follows that any mixed derivative can be realized for appropriately
  chosen (higher order) directional derivatives, therefore the $k$-jet
  generating property is satisfied. The argument works for any $y\in \I_o$ and
  $k\in \N$, therefore, the jet generating property holds.
  
\end{proof}
 
\begin{remark}[The jet generating property of curvature vector fields and their
  derivatives]
  \label{rem:k_jet_gen}
  One can easily show that in the 2-dimensional case, at any point
  $y\in \I_o$, the set of the curvature vector field and its derivatives up to
  order $k$ contains $k+1$ linearly independent $k$-jet, therefore this set
  has the $k$-jet generating property.  In the higher dimensional cases, from
  the argument of Proposition \ref{prop:Funk_generating} using 2-dimensional
  planes, one can obtain that for any point $y\in \I_o$ and any direction
  $v\in T_y(\I_o)$, the curvature vector fields and their derivatives up to
  order $k$ can be used to express the directional derivatives
  \eqref{eq:derivatives}. From formula \eqref{eq:mixed_der} one obtains that
  any $k$\ts{th} order derivative can be obtained by the derivatives of the
  curvature vector fields up to order $k$, that is the set
  \begin{equation}
    \label{eq:k_jet_curv}
    \left\{\mathcal R_{i j}, \ \nabla_{p_1} \! \mathcal R_{i j},  \ \dots,   \ 
      \nabla_{p_1 \dots p_k} \! \mathcal R_{ij} \ |  \ 1 \leq i,j,p_1
      \dots p_k  \leq  n \right\} \subset \X{\I_o}
  \end{equation}
  has the $k$-jet generating property.
\end{remark}

\bigskip

\section{The Funk-perturbed Finsler metrics}
\label{sec:4}

In this section we investigate the holonomy group of a Finsler metric
perturbed with the standard Funk metric.  We present some technical properties
of the Funk deformation which are essential in the proof of Theorem
\ref{main_thm_2}.

Let $(M, F)$ be a Finsler manifold and $x_0\in M$ be a fixed point. We can
chose an $x_0$-centered coordinate system $(U, x)$ such that
$x(U)\subset \B^n$. The associated coordinate system on $TM$ will be denoted
by $( \pi^{-1}(U), \chi=(x, y))$.  We also consider a bump function
$\psi\colon M \to \R$, such that $supp(\psi) \subset U$ and
$\psi|_{\tilde{U}} =1$ for some open neighbourhood $\tilde{U} \subset U$ of
$x_0$. We denote by $\bar{\psi}:= \psi \circ \pi$ the pull-back of $\psi$ by
the projection $\pi$.

Using the standard Funk norm function $F_{\B^n}$, we introduce the Finsler
norm $\bar{F}: TM \to \R$ by the formula
\begin{equation}
  \label{eq:bar_F}
  \bar{F}^2 = \psi \cdot (F_{\B^n} \circ \chi) ^2 + (1-\psi) \cdot F^2.
\end{equation}
We remark that $\bar{F}$ is the pull-back of the standard Funk norm function
on $\pi^{-1}(\tilde{U})$.

\vspace{6pt}

Using \eqref{eq:bar_F} we define a smooth perturbation of the Finsler
function $F$ as a 1-parameter family of functions $F_t$, where
\begin{equation}
  \label{def:pert}
  F_t^2 = (1-t)F^2 + t \bar{F}^2, \quad t\in [0,1].
\end{equation}
Then $F_t$ is a 1-parameter family of Finsler metrics. Indeed, $F$ and
$\bar{F}$ are positively 1-homogeneous continuous function, smooth on
$\T M$, therefore $F_t$ verifies these properties.  Moreover, taking the
squares in \eqref{def:pert} ensures that the bilinear form 
\begin{displaymath}
  g_{ij}^t = (1-t) \, g_{ij} + t \, \bar{g}_{ij} , \quad t\in [0,1]
\end{displaymath}
of $F_t$ is positive definite as well.

\medskip

\begin{proposition}
  \label{property:alg_exp}
  Any element of the infinitesimal holonomy algebra $\ihol_{x_0} (F_t)$
  can be expressed as an algebraic fraction of polynomials in $t$ whose
  coefficients are determined by $j^{k}_{x_0}F$ and $j^{k}_{x_0}\bar{F}$
  for some $k\in \N$.
\end{proposition}
\begin{proof}
  The geodesic coefficients $G^i_t$, $i=1, \dots, n$ of $F_t$ can be
  calculated in terms of $j^3_{x_0}(F_t)$, therefore in terms of $t$,
  $j^3_{x_0} (F)$, and $j^3_{x_0}(\bar{F})$. More precisely, their expressions
  are algebraic fractions of polynomials in $t$ whose coefficients are
  determined by the third order jets of $F$ and $\bar{F}$.  Similarly, the
  curvature vector fields of $F_t$ can be expressed as algebraic fractions of
  polynomials in $t$ whose coefficients are determined by $j^5_{x_0}(F)$ and
  $j^5_{x_0}(\bar{F})$.  More generally, using Remark \ref{remark:jet}, the
  value of $k$\ts{th} order covariant derivatives and $k$\ts{th} successive
  Lie brackets of curvature vector fields of $F_t$ can be expressed as
  algebraic fractions of polynomials in $t$ whose coefficients are determined
  by $j^{k+5}_{x_0}(F)$ and $j^{k+5}_{x_0}(\bar{F})$.

\end{proof}

\medskip
 
\begin{proposition}
  \label{prop:generating_at_t}
  For any $y_0 \in \I_o$ the set of parameters $t\in [0,1]$, where the 3-jet
  generating property of the infinitesimal holonomy algebra
  \begin{math}
    \ihol_{x_0} (F_t) \subset \X{\I_{x_0}^{\, t}}
  \end{math}
  of the Funk perturbation \eqref{def:pert} is not satisfied, is finite.
\end{proposition}

\begin{proof}
  Let us suppose that $y_0\in \mathcal{I}_{x_0}^{\, t} \subset T_{x_0}M$ for every
  $t \in [0,1]$. If not, then we can consider
  \begin{math}
    \widetilde{F}_t(x,y) := 
    F_t(x,y)/F_t(x_0,y_0)
  \end{math} 
  which is just a rescaling with a constant for any given $t$, therefore it
  doesn't affect the jet generating property on the indicatrix.
  
  From Proposition \ref{prop:Funk_generating} we know that the Funk metric
  has the jet generating property, therefore for $t=1$, there are vector
  fields
  \begin{equation}
    \label{eq:W_funk}
    \left\{ W_1,\ldots, W_l \right\} \in \ihol_{x_0} (F_{t=1}),
  \end{equation}
  $l:=\dim (J^3_{y_0}(\X{\I_{x_0}^{\, t=1}}))$ in the infinitesimal holonomy algebra, such
  that any 3\ts{rd} order jet at $y_0$ can be realized with their
  combination.  Those vector fields are linear combination of curvature
  vector fields, their derivatives and their Lie brackets. Let us consider
  them for any $F_t$, $t\in [0,1]$. We get
  a set in the infinitesimal holonomy algebra of $F_t$ at $x_0$: 
  \begin{equation}
    \label{eq:W}
    \left\{ W_1(t),\ldots, W_l(t) \right\} \in \ihol_{x_0} (F_t).
  \end{equation}
  Using Proposition \ref{property:alg_exp}, these vector fields are
  algebraic fractions of polynomials in $t$ whose coefficients are
  determined by $j^k_{x_0}(F)$ and $j^k_{x_0}(\bar{F})$ for some $k\in \N$.
  It follows that the determinant of the $l\times l$ matrix composed by the
  3\ts{rd} order jet coordinates of \eqref{eq:W} at $y_0$:
  \begin{equation}
    \label{eq:P_t}
    \mathcal P_t:=
    \det \left(
      \begin{matrix}
        j^3_{y_0} (W_1(t))
        \\
        \vdots
        \\
        j^3_{y_0} (W_l(t))
      \end{matrix}
    \right),
  \end{equation}
  is an algebraic fraction of polynomials in $t$ whose coefficients are
  determined by $j^3_{x_0}(F)$ and $j^3_{x_0}(\bar{F})$ for some $k\in \N$
  with $\mathcal P_{t=1} \not \equiv 0$.  Since every non-trivial
  polynomial has finitely many roots, $\mathcal P_t$ can only be zero at
  finitely many values $t\in[0,1]$. By continuity, there is a neighbourhood
  of $y_0$ where this property is satisfied.
\end{proof}

\bigskip

\section{Density of Finsler metrics with infinite dimensional
  holonomy group}
\label{sec:5}

\noindent
In this section we prove Theorems  \ref{thm:no_transitive_finite} and  
\ref{main_thm_2}.  

\subsection{Proof of Theorem \ref{thm:no_transitive_finite}.}
 
We prove the theorem by contradiction: let us suppose that
$\mathfrak{g} \subset \X{U}$ is a \emph{finite} dimensional Lie algebra on an
$n$-dimensional manifold $U$ and it is generating the third order jets at
$x_0\in U$. As before, the last property means that the 3-jet projection
\begin{math}
  \mathfrak{g} \to J^3_{x_0} (\X{U})
\end{math}
is onto.

We remark that a manifold with a finite dimensional Lie algebra of vector
fields with locally transitive action is real analytic (in the sense that
there exist a real-analytic atlas such that the vector fields of the Lie
algebra are real-analytic). Indeed, a finite dimensional Lie algebra
generates a Lie group, and one can provide an analytic atlas on this Lie
group so that the group multiplication is analytic. Moreover, local Lie
subgroups are real analytic submanifolds, because they are images of the
exponential map. For more about Lie groups and related topics, we refer to
\cite{Kolar_Michor_Slovak_1993, Pontryagin_1954}.  It follows that locally,
a manifold with transitive action of a Lie group is the factor group of the
Lie group by the stabilizer of one element, which is a local Lie subgroup,
then it is also analytic.

We say that the order of singularity of an element $v\in \mathfrak{g}$ at
$x_0$ is $k \in \N$, noted as $\mathcal O_{x_0}(v)=k$, if the value and all
partial derivatives up to order $k$ at $x_0$ are zero, and $v$ has a
non-vanishing $(k+1)$st order derivative.  Each nonzero element has finite
order by analyticity.  Let us consider the set
$\mathfrak{g}_1 \subset \mathfrak{g}$ with order of singularity at least
one, that is
  \begin{displaymath}
    \mathfrak{g}_1:=\Big\lbrace v\in \mathfrak{g}\quad
    \Big|\quad v(x_0)=0, \ \frac{\partial v}{\partial x_{i}}(x_0)=0, 
    \quad  i=1, \ldots,  n  \Big\rbrace.
  \end{displaymath}
  It is easy to see that $\mathfrak{g}_1$ is a Lie subalgebra of
  $\mathfrak{g}$. Indeed, the $j$th component of the commutator of two vector
  fields $v,u\in \mathfrak{g}_1$ is given by
  \begin{displaymath}
    [ v, u]_j =\sum_i \left(\tfrac{ \partial v_j }{\partial x_i} u_i  -
      \tfrac{ \partial u_j }{\partial x_i} v_i\right)
\end{displaymath}
and has singularity of order at least two at $x_0$. Actually, for any two
vector fields $V, U$ from $\mathfrak{g}_1$ such that $V$ has order of
singularity $k$ and $U$ has order of singularity $m$ their commutator has
order of singularity $k+m$.
   
Since $\mathfrak{g}$ is finite dimensional, so is $\mathfrak{g}_1$. It follows
that the order of singularity is bounded on $\mathfrak{g}_1$. Indeed, if not,
then there would be a sequence of vectors in $\mathfrak{g}_1$ with strictly
monotone increasing order of singularity at $x_0$ which would produce an
infinite number of linearly independent elements which is impossible.

Let $v_1\in \mathfrak{g}_1$ be a non-zero element with maximal order
$\mathcal O_{x_0}(v_1)=k$ of singularity at $x_0$. Using the 3-jet generating
property, we have $k \geq 2$. Then, for any $v\in \mathfrak{g}_1$ we have
$[v,v_1]\in \mathfrak{g}_1$ and
\begin{math}
  \mathcal O_{x_0}([v,v_1]) > k.
\end{math}
Since $k$ is maximal, it follows that $[v,v_1]=0$, and it shows that $v_1$
commutes with every elements of $\mathfrak{g}_1$.

On the other hand, it is possible to choose a point $\hat{x}_0 \in U$ in a
neighbourhood of $x_0$ such that $v_1(\hat{x}_0) \neq 0$ and $\mathfrak{g}_1$
has the 1-jet generating property at $\hat{x}_0$.  It follows that one can
choose an element $v_2 \in \mathfrak{g}_1$ such that $v_2(\hat{x}_0)=0$ but
$\mathcal D_{v_1} v_2 \neq 0$. Therefore the commutator $[v_1,v_2]$ at
$\hat{x}_0$ is non zero. This is a contradiction since $v_1$ is an element
which commutes with every element of $\mathfrak{g}_1$. Theorem
\ref{thm:no_transitive_finite} is proved.

\medskip

\subsection{Proof of Theorem \ref{main_thm_2}.} \

\noindent
Let $\F$ be the set of $C^\infty$-smooth Finsler metrics on a given manifold
$M$ and let us consider the subset $\widetilde{\F} \subset \F$ characterized
by the following property: $\tilde{F} \in \widetilde{\F}$ if and only if there
exists a point $x_0 \in M$ such that the curvature vector fields and their
derivatives
\begin{equation}
  \label{eq:3_jet_curv}
  \left\{\mathcal R_{i j}, \ \nabla_{k} \mathcal R_{i j},  \
    \nabla_{k l}  \mathcal R_{ij},   \ 
    \nabla_{klh}  \mathcal R_{ij} \ |  \ 1 \leq i,j,k,l,h
    \leq  n \right\} \subset \X{\I_{x_0}},
\end{equation}
up to order 3 has the $3-$jet generating property at least at one point of the
indicatrix at $x_0$.  Then
\begin{enumerate}
\item [\emph{i)}] the holonomy group at $x_0$ of any $\tilde{F} \in \widetilde{\F} $
  is infinite dimensional,
\item [\emph{ii)}] the set $\widetilde{\F}$ is dense in $\F$ with respect to
  the $C^m$ topology for each $m\geq 8$ (in fact for any $m\ge 0$),
\item [\emph{iii)}] the set $\widetilde{\F}$ is open in $\F$ with respect the
  the $C^{\widetilde{m}}$ topology for $\widetilde{m}\geq 8$.
\end{enumerate}
Indeed, \emph{i)} follows from Theorem \ref{thm:no_transitive_finite}: the
infinitesimal holonomy algebra $\ihol_{x_0}(\tilde{F})$ is infinite
dimensional, consequently, the holonomy algebra and the holonomy group of
$\tilde{F}$ at $x_0$ are infinite dimensional.
   
In order to show \emph{ii)}, let us consider the Funk-perturbation $F_t$
given by \eqref{def:pert}, a point $x_0$ in $M$ and a point $y_0$ of the
indicatrix at $x_0$.  By Proposition \ref{prop:generating_at_t} there
exists a sufficiently small $t>0$ such that the curvature vector fields and
their derivatives up to order 3 has the $3-$jet generating property at
$y_0$.  For sufficiently small $t$ the metric $F_t$ is sufficiently close
to $F$ in $C^m$-topology.

In order to prove \emph{iii)}, we observe that the jet-generating condition is
an open condition, so if it is fulfilled at $y_0\in \mathcal{I}_{x_0}$, it is
fulfilled at any point $y_1\in \mathcal{I}_{x_1}$ sufficiently close to $y_0$
in $C^{m\ge 8}$-topology on $TM$.

Theorem \ref{main_thm_2} is proved.

\begin{remark}
  \label{rem:germs}
  In the proof of Theorem \ref{main_thm_2}, it was showed that for a
  convenient perturbation the infinitesimal holonomy algebra at a point is
  infinite-dimensional. This remains valid mircrolocally and on the level of
  germs.
\end{remark}


\bigskip



\begin{thebibliography}{10}
\bibitem{Berger_1955} M.~Berger.  \newblock Sur les groupes d'holonomie
  homog\`ene des vari\'et\'es \`a connexion affine et des vari\'et\'es
  riemanniennes.  \newblock {\em Bull. Soc. Math. France}, 83: 279--330, 1955.

\bibitem{Berwald_1926} L.~Berwald. \newblock Untersuchung der Krümmung
  allgemeiner metrischer {R}äumen auf {G}rund des in inher herrschenden
  {P}arallelism.  \newblock {\em Math. Z.}, 25: 40--73, 1926.

\bibitem{Borel_Lichn_1952} A.~Borel and A.~Lichnerowicz.  \newblock Groupes
  d'holonomie des vari\'et\'es riemanniennes.  \newblock {\em
    C. R. Acad. Sci. Paris}, 234: 1835--1837, 1952.

\bibitem{ChernShen2005} S.-S.~Chern, Z.~Shen. \newblock Riemann-Finsler
  geometry.  \newblock {\em Nankai Tracts in Mathematics,} \newblock 42
  6. World Scientific Publishing Co. Pte. Ltd., Hackensack, NJ, 2005. x+192
  pp.

\bibitem{Finsler_1918} P.~Finsler. \newblock Ueber {K}urven und {F}lächen in
  allgemeinen {R}äumen. \newblock {\em Doctoral thesis}, 1918.

\bibitem{Hubicska_Muzsnay_2017} B.~{Hubicska} and Z.~{Muzsnay}.  \newblock
  {The holonomy groups of projectively flat Randers two-manifolds of constant
    curvature.}  \newblock {\em {arXiv:1805.05216}}, 2018.

\bibitem{Hubicska_Muzsnay_tangent_2017} B.~{Hubicska} and Z.~{Muzsnay}.
  \newblock {Tangent Lie Algebra of a Diffeomorphism Group and Application to
    Holonomy Theory.}  \newblock {\em {The Journal of Geometric Analysis,
      (arXiv:1805.05265)}}, Jan 2019.


\bibitem{Ivanov_Lytchak_2019} S.~Ivanov and A.~Lytchak.  \newblock Rigidity
  of Busemann convex Finsler metrics. {\em Comment. Math. Helv.},
  94:855--868, 2019.
  
\bibitem{Kolar_Michor_Slovak_1993} I.~Kolar, Peter W.~Michor and J.~Slovak.
  \newblock Natural Operations in Differential Geometry. {\em Springer-Verlag
    Berlin Heidelberg}, 1993.

\bibitem{Kozma_2000} L.~Kozma, On holonomy groups of Landsberg manifolds. {\em
    Tensor (N.S.)}, 62: 87--90, 2000.

\bibitem{relativity} V.S.~Matveev. \newblock Geodesically equivalent metrics
  in general relativity. \newblock {\em J. Geom. Phys. } 62:675--691, 2012.

\bibitem{Matveev} V.S.~Matveev. \newblock On ``All regular Landsberg metrics
  are always Berwald" by Z. I. Szabo. \newblock {\em Balkan Journ. Geom.,}
  14:50--52, 2009.
	
\bibitem{Ma2} V. S.~Matveev.  \newblock {Riemannian metrics having common
    geodesics with Berwald metrics}, \newblock {\em Publ. Math. Debr.,}
  74:405--416, 2009.
		
\bibitem{Matveev1} V.S.~Matveev and M.~Troyanov.  \newblock {The
    Binet-Legendre metric in Finsler geometry}.  \newblock {\em Geometry \&
    Topology}, 16:2135--2170, 2012.

\bibitem{Muzsnay_Nagy_2015} Z.~Muzsnay and P.~T. Nagy. \newblock
  Projectively flat {F}insler manifolds with infinite dimensional
  holonomy. \newblock {\em Forum Math.}, 27: 767--786, 2015.

\bibitem{Muzsnay_Nagy_max_2015} Z.~Muzsnay and P.~T. Nagy.  \newblock
  Finsler 2-manifolds with maximal holonomy group of infinite dimension.
  \newblock {\em Differential Geom. Appl.}, 39: 1--9, 2015.

\bibitem{Olver} P. ~Olver.  \newblock Equivalence, invariants, and
  symmetry. \newblock {\em Cambridge University Press}, Cambridge,
  1995. xvi+525 pp. ISBN: 0-521-47811-1

\bibitem{Pontryagin_1954} L.S.~Pontryagin. \newblock Nepreryvnye
  gruppy. Translation: Topological Groups. \newblock {\em Gordon and Breach.},
  1966.

\bibitem{Riemann_1854} B.~Riemann. \newblock Über die {H}ypothesen, welche der
  {G}eometrie zu {G}runde liegen. \newblock Göttingen, 1854.

\bibitem{Simons} J.~Simons. \newblock On transitivity of holonomy
  systems. \newblock {\em Annals of Math.,} 76: 213--234, 1962.

\bibitem{Szabo_1981} Z.I.~Szab\'o.  \newblock Positive definite {B}erwald
  spaces. {S}tructure theorems on {B}erwald spaces.  \newblock {\em Tensor
    (N.S.)}, 35(1): 25--39, 1981.
		
\bibitem{Vi} Cs.~Vincze.  \newblock A new proof of Szab\'o's theorem on the
  Riemann metrizability of Berwald manifolds. \newblock {\em Acta
    Math. Acad. Paedagog. Nyhazi,} 21:199--204, 2005.
\end{thebibliography}
\end{document}